\documentclass[12pt,reqno]{article}
\usepackage[usenames]{color}
\usepackage{amssymb}
\usepackage{graphicx}
\usepackage{amscd}
\usepackage{epsfig}
\usepackage[colorlinks=true,
linkcolor=webgreen, filecolor=webbrown,
citecolor=webgreen]{hyperref}
\definecolor{webgreen}{rgb}{0,.5,0}
\definecolor{webbrown}{rgb}{.6,0,0}
\usepackage{color}
\usepackage{fullpage}
\usepackage{float}
\usepackage{psfig}
\usepackage{graphics,amsmath,amssymb}
\usepackage{amsthm}
\usepackage{amsfonts}
\usepackage{latexsym}
\usepackage{epsf}
\newcommand{\seqnum}[1]{\href{http://www.research.att.com/cgi-bin/access.cgi/as/~njas/sequences/eisA.cgi?Anum=#1}{\underline{#1}}}

\begin{document}
\begin{center}
\vskip 1cm {\LARGE\bf A characterization of regular tetrahedra in
$\Bbb Z^3$ }

\vskip 1cm \large Eugen J. Ionascu\\
\vskip 1cm

{\small
Department of Mathematics\\
Columbus State University\\
Columbus, GA 31907, US \\
{\it Honorific Member of the \\ Romanian Institute of Mathematics
``Simion Stoilow"} \\
\href{mailto:ionascu_eugen@colstate.edu}{\tt ionascu\_eugen@colstate.edu}} \\
\end{center}

\centerline{December $23^{rd}$, 2007}

\vskip .2in

\begin{abstract}
In this note we characterize all regular tetrahedra whose vertices
in $\mathbb{R}^3$ have integer coordinates. The main result is a
consequence of the characterization of all equilateral triangles
having integer coordinates (\cite{rceji}). Previous work on this
topic began in \cite{eji1}. We will use this characterization to
point out some corollaries. The number of such tetrahedra whose
vertices are in the finite set $\{0,1,2,...,n\}^3$, $n\in\mathbb N$,
is related to the sequence A103158 in the Online Encyclopedia of
Integer Sequences (\cite{OL}).
\end{abstract}

\newtheorem{theorem}{Theorem}[section]
\newtheorem{proposition}{Proposition}[section]
\newtheorem{corollary}{Corollary}[section]
\newtheorem{lemma}{Lemma}[section]
\newtheorem{definition}{Definition}[section]
\newtheorem{problem}{Problem}[section]
\section{Introduction}

In this paper we give a solution to the following system of
Diophantine equations:

\begin{equation}\label{selling}
\begin{cases}
\begin{array}{c}
  (x_a-x_b)^2+(y_a-y_b)^2+(z_a-z_b)^2 =(x_a-x_c)^2+(y_a-y_c)^2+(z_a-z_c)^2
  \\ \\
  (x_a-x_b)^2+(y_a-y_b)^2+(z_a-z_b)^2
  =(x_a-x_d)^2+(y_a-y_d)^2+(z_a-z_d)^2\\ \\
(x_a-x_b)^2+(y_a-y_b)^2+(z_a-z_b)^2
=(x_b-x_c)^2+(y_b-y_c)^2+(z_b-z_c)^2
  \\ \\
  (x_a-x_b)^2+(y_a-y_b)^2+(z_a-z_b)^2
  =(x_b-x_d)^2+(y_b-y_d)^2+(z_b-z_d)^2\\ \\
(x_a-x_b)^2+(y_a-y_b)^2+(z_a-z_b)^2
  =(x_c-x_d)^2+(y_c-y_d)^2+(z_c-z_d)^2
\end{array}
\end{cases}
\end{equation}
\noindent where $x_a$, $y_a$, $z_a$, $x_b$, $y_b$, $z_b$, $x_c$,
$y_c$, $z_c$, $x_d$, $y_d$, $z_d$ are in $\mathbb Z$. If
$A(x_a,y_a,z_a)$, $B(x_b,y_b,z_b)$, $C(x_c,y_c,z_c)$, and
$D(x_d,y_d,z_d)$ are considered in $\mathbb R^3$ the Diophantine
system (\ref{selling}) represents exactly the condition that makes
$ABCD$ a regular tetrahedron with vertices having integer
coordinates. Clearly (\ref{selling}) is invariant under translations
in $\mathbb R^3$ with vectors having integer coordinates. So,
without loss in generality we are interested in the case, $A=O$,
where $O$ is the origin. Then for each of the faces OBC, OCD and OBD
to be being equilateral triangles it is necessary and sufficient to
solve for instance the following system:

\begin{equation}\label{eqtr}
\begin{cases}
\begin{array}{c}
  x_b^2+y_b^2+z_b^2=x_c^2+y_c^2+z_c^2
  \\ \\
  x_b^2+y_b^2+z_b^2
  =(x_b-x_c)^2+(y_b-y_c)^2+(z_b-z_c)^2.
\end{array}
\end{cases}
\end{equation}

This problem was solved in \cite{rceji}. We recall the following
facts from \cite{rceji} and \cite{eji1}. Every equilateral triangle
in $\mathbb{R}^3$ whose vertices have integer coordinates, after a
translation that brings one of its vertices to the origin, must be
contained in a plane of a normal vector $(a,b,c)$ ($a,b,c\in \Bbb
Z$) satisfying the Diophantine equation

\begin{equation}\label{planeeq}
a^2+b^2+c^2=3d^2 \end{equation}

\noindent with $d\in \Bbb N$. Moreover, $a$, $b$, $c$, and $d$ can
be chosen so that the sides of such a triangle have length of the
form $d\sqrt{2(m^2-mn+n^2)}$ with $m$ and $n$ integers. A more
precise statement is the next parametrization of these triangles.

\begin{theorem}\label{generalpar} Let $a$, $b$, $c$, $d$ be odd
integers such that $a^2+b^2+c^2=3d^2$ and $gcd(a,b,c)=1$. Then for
every $m,n \in\mathbb{Z}$ the points $P(u,v,w)$ and $Q(x,y,z)$ whose
coordinates are given by
\begin{equation}\label{paramone}
\begin{cases}
u=m_um-n_un,\\
v=m_vm-n_vn,\\
w=m_wm-n_wn, \\
\end{cases}
\ \ \ \text{and} \ \ \
\begin{cases}
x=m_xm-n_xn,\\
y=m_ym-n_yn,\\
z=m_zm-n_zn,
\end{cases}\ \
\end{equation}
with
\begin{equation}\label{paramtwo}
\begin{array}{l}
\begin{cases}
m_x=-\frac{1}{2}[db(3r+s)+ac(r-s)]/q,\ \ \ n_x=-(rac+dbs)/q\\
m_y=\frac{1}{2}[da(3r+s)-bc(r-s)]/q,\ \ \ \ \  n_y=(das-bcr)/q\\
m_z=(r-s)/2,\ \ \ \ \ \ \ \ \ \ \ \ \ \ \ \  \ \ \ \ \  \ \ \ \ \ \ n_z=r\\
\end{cases}\\
\ \ \ \text{and}\ \ \\
\begin{cases}
m_u=-(rac+dbs)/q,\ \ \ n_u=-\frac{1}{2}[db(s-3r)+ac(r+s)]/q\\
m_v=(das-rbc)/q,\ \ \ \ \   n_v=\frac{1}{2}[da(s-3r)-bc(r+s)]/q\\
m_w=r,\ \ \ \ \ \ \ \ \ \ \ \ \ \ \ \ \ \ \  n_w=(r+s)/2\\
\end{cases}
\end{array}
\end{equation}
where $q=a^2+b^2$ and $(r,s)$ is a suitable solution of
$2q=s^2+3r^2$ that makes all the numbers in (\ref{paramtwo})
integers, together with the origin (O(0,0,0)) forms an equilateral
triangle in $\Bbb Z^3$ contained in the plane having equation
$${\cal
P}_{a,b,c}:=\{(\alpha,\beta,\gamma)\in \mathbb {R}|\
a\alpha+b\beta+c\gamma=0\}$$ and having sides-lengths equal to
$d\sqrt{2(m^2-mn+n^2)}$.

Conversely, there exists a choice of the integers $r$ and $s$ such
that given an arbitrary equilateral triangle in $\mathbb{R}^3$ whose
vertices, one at the origin and the other two having integer
coordinates and contained in the plane ${\cal P}_{a,b,c}$, then
there also exist integers $m$ and $n$ such that the two vertices not
at the origin are given by (\ref{paramone}) and (\ref{paramtwo}).
\end{theorem}

\begin{center}\label{figure0}
\epsfig{file=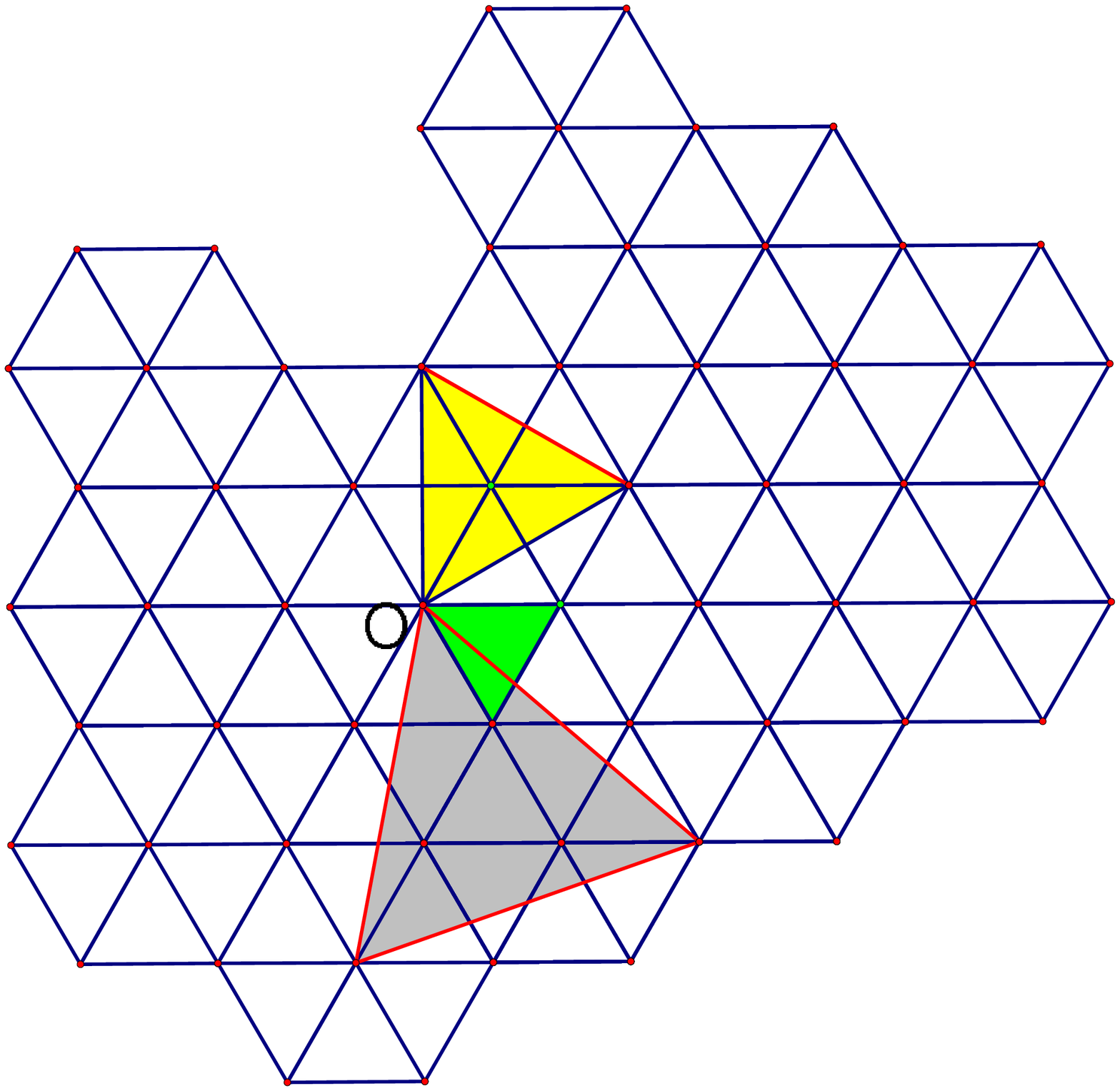,height=3in,width=3in}
\vspace{.1in}\\
{\it {\small Figure 1: Lattice generated by (\ref{paramone}) with
$m,n\in \mathbb Z$ \\ in the plane ${\cal
P}_{a,b,c}:=\{(\alpha,\beta,\gamma)\in \mathbb {R}|\
a\alpha+b\beta+c\gamma=0\}$}  } \vspace{.1in}
\end{center}

In Figure~1 we are showing a few triangles that can be obtained,
using (\ref{paramone}), for various values of $m$ and $n$. Out of
all these equilateral triangles only a few are faces of regular
tetrahedra with integer coordinates. The restriction comes from the
following characterization given as Proposition~5.3 in \cite{eji1}.

\begin{proposition}\label{tetrahedra}
A regular tetrahedron having side lengths $l$ and with integer
coordinates exists, if and only if $l=m\sqrt{2}$ for some $m\in
\mathbb N$.
\end{proposition}
Given an $m\in \Bbb N$, we are interested in finding a
characterization of all such tetrahedra having side lengths equal to
$m\sqrt{2}$. Let us denote this class of objects by $\mathrm{T}_m$
and its subset of tetrahedra that have the origin as a vertex by
$\mathrm{T}_m^0$. Obviously we have $\mathrm{T}_m= \underset {v\in
\Bbb Z^3} {\cup } (\mathrm{T}_m^0+v)$ but this union is not a
partition. In what follows we will use the notation
$\mathrm{T}:=\underset {m\in \Bbb N} {\cup } (\mathrm{T}_m)$. This
latter union is, on the other hand, a partition.

Next we are going to concentrate on $\mathrm{T}_m^0$. It turns out
that $\mathrm{T}_m^0$ is numerous if $m$ is divisible by many prime
factors of the form $6k+1$, $k\in \Bbb N$, as one can see from the
following. First we recall Euler's $6k+1$ theorem (see \cite{r}, pp.
568 and \cite{g}, pp. 56).
\begin{proposition}\label{ejtypet} An integer $t$ can be written as $s^2-sr+r^2$ for some
$s,r\in \mathbb Z$ if and only if in the prime factorization of $t$,
$2$ and the primes of the form $6k-1$ appear to an even exponent.
\end{proposition}

The following result seems to be known but we could not find a
reference for it.
\begin{proposition} \label{numberofrepres}
Given $k=3^{\alpha}(a^2)b$ where $a$ does not contain any prime
factors of the form $6 \ell+1$, $b$ is the product of such primes,
say $b=p_1^{r_1}p_2^{r_2}\cdots p_j^{r_j}$, then the number of
representations $k=m^2-mn+n^2$, counting all possible pairs $(m,n)$
of integers that satisfy this equation, is equal to
$$6(r_1+1)(r_2+1)\cdots (r_j+1).$$
\end{proposition}

For a proof of this proposition one can use arguments similar to
those in the proof of Gauss's Theorem about the number of
representations of a number as sums of two squares of integers
(\cite{r}). In this case one has to replace Gaussian integers with
Eisenstein integers, i.e. the ring of numbers of the form
$m+n\omega$, $m,n\in \Bbb Z$, where $\omega=\frac{-1+i\sqrt{3}}{2}$.
Related to Proposition~\ref{numberofrepres}, we found a conjecture
in a paper posted on the mathematics archives (see \cite{n}). This
article refers to the number of representations of a number as
$m^2+mn+n^2$ with $m,n\in \Bbb N$.

\section{Main results}

Let us begin by refining the argument used in the proof of
Proposition~5.3 in \cite{eji1}. That proposition referred to the
following figure representing a regular tetrahedron whose vertices
are given by the origin, $P(x,y,z)$, $Q(u,v,w)$ and $R$:

\begin{center}\label{figure1}
\epsfig{file=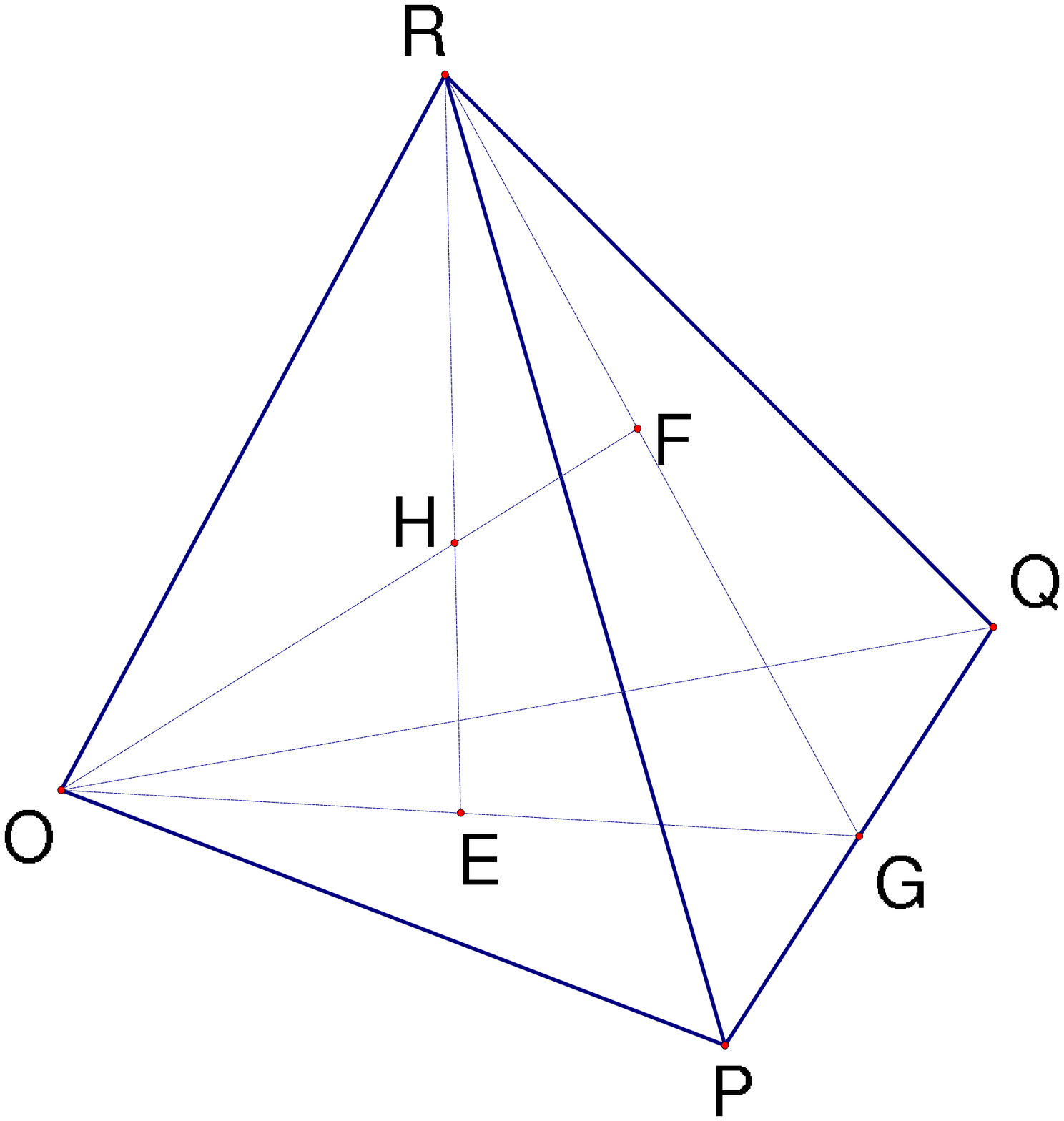,height=3in,width=3in}
\vspace{.1in}\\
{\it Figure 2: Regular tetrahedron} \vspace{.1in}
\end{center}

By the characterization of equilateral triangles in
Theorem~\ref{generalpar} we may assume that the coordinates of $P$
and $Q$ are given by (\ref{paramone}) and (\ref{paramtwo}).

If $E$ denotes the center of the face $\triangle OPQ$. Then from
Theorem~\ref{generalpar} we know that
$\frac{\overset{\rightarrow}{ER}}{|\overset{\rightarrow}{ER}| }=
\frac{(a,b,c)}{\sqrt{a^2+b^2+c^2}}=\frac{1}{\sqrt{3}d}(a,b,c)$ for
some $a,b,c,d,m,n\in\mathbb Z$, $gcd(a,b,c)=1$, $d$ odd satisfying
$a^2+b^2+c^2=3d^2$ and $l=d\sqrt{2(m^2-mn+n^2)}$. Let us denote
$m^2-mn+n^2$ by $\zeta(m,n)$. The coordinates of $E$ are
$(\frac{u+x}{3},\frac{y+v}{3},\frac{z+w}{3})$.

From the Pythagorean theorem one can find easily that $RE=l
\sqrt{\frac{2}{3}}$. Since
$\overset{\rightarrow}{OR}=\overset{\rightarrow}{OE}+\overset{\rightarrow}{ER}$,
the coordinates of $R$ must be given by

$$\left(\frac{u+x}{3}\pm
l\sqrt{\frac{2}{3}}\frac{1}{\sqrt{3}d}a,\frac{y+v}{3}\pm
l\sqrt{\frac{2}{3}} \frac{1}{\sqrt{3}d}b,\frac{z+w}{3}\pm
l\sqrt{\frac{2}{3}}\frac{1}{\sqrt{3}d}c\right)$$ or
$$\left(\frac{u+x}{3}\pm \frac{2a\sqrt{\zeta(m,n)}}{3},\frac{y+v}{3}\pm \frac{2b\sqrt{\zeta(m,n)}}{3},
\frac{z+w}{3}\pm \frac{2c\sqrt{\zeta(m,n)}}{3}\right).$$ Since these
coordinates are assumed to be integers we see that $\zeta(m,n)$ must
be a perfect square, say $k^2$, $k\in \Bbb N$.

It is worth mentioning that this leads to the Diophantine equation

\begin{equation}\label{eisensteintriple}
m^2-mn+n^2=k^2
\end{equation}

\noindent whose positive solutions are known as Eisenstein triples,
or Eisenstein numbers.  A primitive solution of
(\ref{eisensteintriple}) is one for which $gcd(m,n)=1$. These
triples can be characterized in a similar way as the Pythagorean
triples.

\begin{theorem}\label{chartheoremsixty} Every primitive solution of the Diophantine equation
(\ref{eisensteintriple}) is in one of the two forms:

\begin{equation}\label{parforsixty}
\begin{array}{c}
\displaystyle m=v^2-u^2,\ n=2uv-u^2 \ and\ k=u^2-uv+v^2, \ with \ v>u, \  or\\ \\
\displaystyle m=2uv-u^2,\ n=2uv-v^2 \ and\ k=u^2-uv+v^2, \ with\
2v>u>v/2,
\end{array}
\end{equation}

\noindent where $u,v\in \Bbb N$, $gcd(u,v)=1$, and $u+v\not \equiv
0$ (mod 3). Conversely, every triple given by one of the
alternatives in (\ref{parforsixty}) is a primitive solution of
(\ref{eisensteintriple}).
\end{theorem}

\noindent We leave the proof of this theorem to the reader.

Next, let us introduce the notation $\Omega(k):=\{(m,n)\in \Bbb Z^2
:\zeta(m,n)=k^2\}$. If the primes dividing $k$ are all of the form
$p\equiv 2$ (mod 3) then we have simply
$$\Omega(k)=\{(k,0), (k,k),
(0,k), (-k,0), (-k,-k), (0,-k)\}$$ but in general this set can be a
much larger as one can see from Proposition~\ref{numberofrepres}.
For instance if $k=7$ then one can check the nontrivial
representation $49=\zeta(8,3)$. For each solution $(m,n)\in
\Omega(k)$, in general, one can find eleven more by applying the
following transformations: $(m,n)\overset{\tau_1}{\rightsquigarrow}
(m-n,m)$, $(m,n)\overset{\tau_2}{\rightsquigarrow} (m,n-m)$, then
their permutations $(m,n)\overset{\tau_3}{\rightsquigarrow} (n,m)$,
$(m,n)\overset{\tau_4}{\rightsquigarrow} (m, m-n)$,
$(m,n)\overset{\tau_5}{\rightsquigarrow} (n-m,n)$, and finally the
reflections into the origin of all the above maps
$(m,n)\overset{\tau_6}{\rightsquigarrow} (-m,-n)$,
$(m,n)\overset{\tau_7}{\rightsquigarrow} (n-m,-m)$,
$(m,n)\overset{\tau_8}{\rightsquigarrow} (-m,m-n)$,
$(m,n)\overset{\tau_9}{\rightsquigarrow} (-n,-m)$,
$(m,n)\overset{\tau_{10}}{\rightsquigarrow} (-m, n-m)$,
$(m,n)\overset{\tau_{11}}{\rightsquigarrow} (m-n,-n)$.

Hence, the coordinates of $R$ depend on $(m,n)\in \Omega(k)$ and
two possible choices of signs:

\begin{equation}\label{fourthpoint}
\begin{array}{c}
 \displaystyle R= \displaystyle \left(
\frac{\begin{array}{c}
        (m_x+m_u)m \\
        -(n_x+n_u)n \\
       \pm 2ak
      \end{array}
}{3},\frac{\begin{array}{c}
        (m_y+m_v)m \\
        -(n_y+n_v)n \\
       \pm 2bk
      \end{array}}{3},
 \displaystyle \frac{\begin{array}{c}
        (m_z+m_w)m \\
        -(n_z+n_w)n \\
       \pm 2ck
      \end{array}}{3} \right),  \ (m,n)\in
\Omega(k).
\end{array}
\end{equation}

We would like to show that for every primitive solution of the
equation (\ref{planeeq}), every $k\in \Bbb N$, and every $(m,n)\in
\Omega(k)$ one has either integer coordinates in (\ref{fourthpoint})
or can choose the signs in order to accomplish this.  A primitive
solution of (\ref{planeeq}) is a solution that satisfies the
conditions $gcd(a,b,c)=1$. We begin with a few lemmas.

\begin{lemma}\label{primitivesol6kplusminus1} Every   primitive
solution $(a,b,c,d)$ of (\ref{planeeq}) must satisfy

(i) $a\equiv \pm 1$ (mod 6),   $b\equiv \pm 1$ (mod 6) and  $c\equiv
\pm 1$ (mod 6),

(ii) $a^2+b^2\equiv 2$ (mod 6).
\end{lemma}

\begin{proof} One can see that all numbers $a$, $b$, $c$ and $d$
must be odd since $gcd(a,b,c)=1$. Then an odd perfect square is
congruent to only $1$ or $3$ modulo $6$.  By way of contradiction if
we have for instance $a \equiv 3$ (mod 6) then $3$ divides
$a^2=3d^2-(b^2+c^2)$. This implies that $3$ divides $b^2+c^2$ which
is possible only if $3$ divides $b$ and $c$. Therefore that would
contradict the assumption $gcd(a,b,c)=1$. Hence, (i) is shown and
(ii) is just a simple consequence of (i).\end{proof}

Let us observe that $(ii)$ in Lemma~\ref{primitivesol6kplusminus1}
implies in particular that $q=a^2+b^2$ is coprime with $3$.

\begin{lemma}\label{mandnequiv}   Let $k\in \Bbb N$ and $(m,n)\in \Omega(k)$.

(a) Then $m$ and $n$ must satisfy $k\equiv \pm (m+n)$ (mod 3).

(b) If $k\equiv 0$ (mod 3) then $m\equiv n\equiv 0$ (mod 3).
\end{lemma}

\begin{proof} (a) This part follows immediately from the identities
$$
\begin{array}{c}
k^2\equiv 4k^2=4m^2-4mn+4n^2=3m^2+(2m-n)^2\equiv (2m-n)^2\equiv
(m+n)^2\ (mod 3).
\end{array}
$$

(b) For the second part let us observe that if $k\equiv 0$ (mod 3)
then by part (a) we must have  $m=3t-n$ for some $t\in \mathbb Z$.
This in turn gives $9k'^2=k^2=m^2-mn+n^2=3(3t^3-3nt+n^2)$ which
shows that $3$ divides $3t^3-3nt+n^2$. So, finally  $n$ must be
divisible by $3$ and then so is $m$. \end{proof}

\begin{center}
\epsfig{file=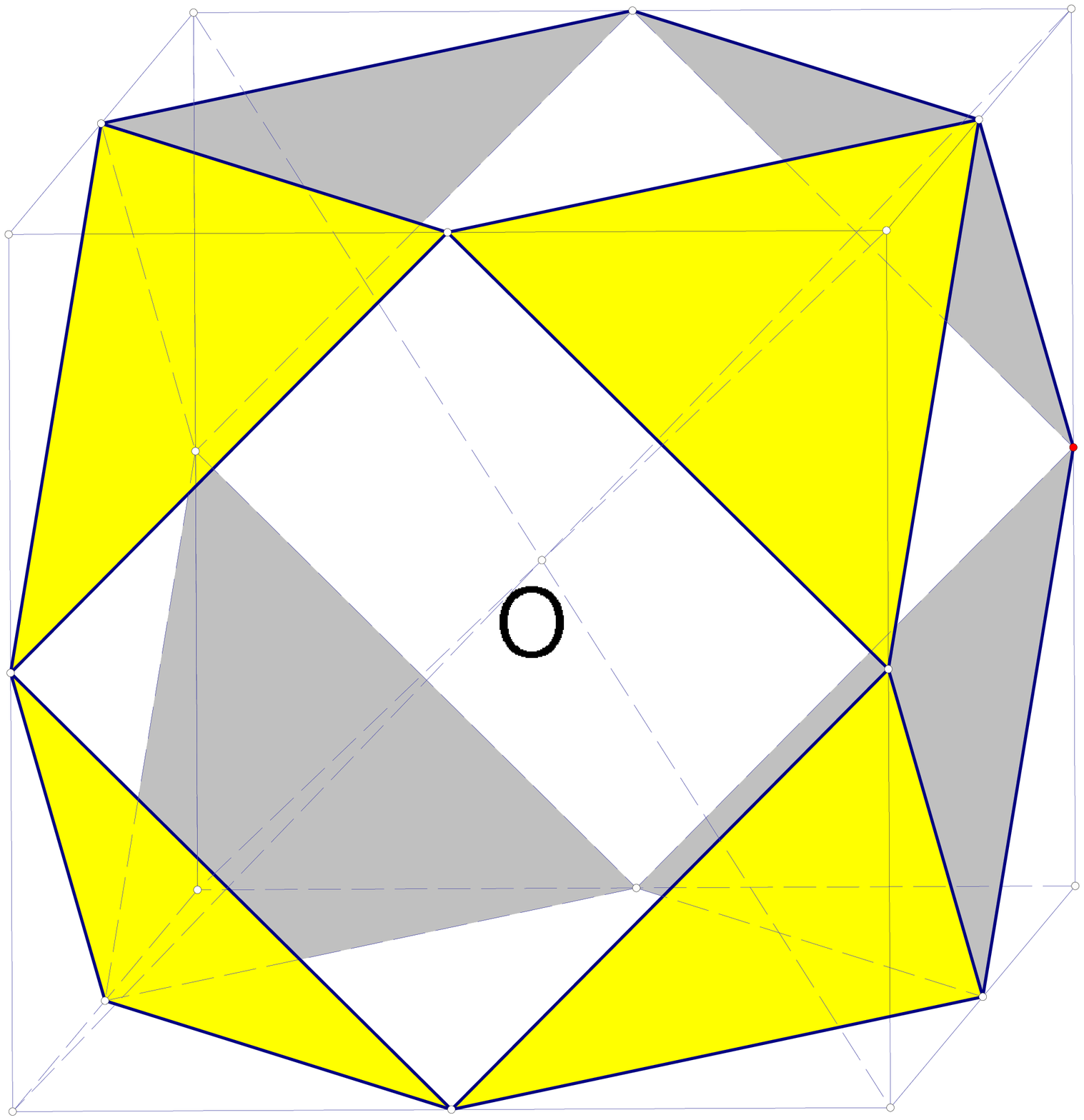,height=2.5in,width=3in}
\vspace{.1in}\\
{\it Figure 3: All regular tetrahedra in $\mathrm{T}_1^0$}
\vspace{.1in}
\end{center}
\begin{lemma}\label{randsequiv} Let $q=a^2+b^2$ with $(a,b,c,d)$ a primitive solution of (\ref{planeeq})  and
$r,s\in \Bbb Z$ a solution of the equation $2q=s^2+3r^2$. Then
$1\equiv a^2\equiv b^2\equiv c^2\equiv s^2$ (mod 3).
\end{lemma}

\begin{proof} Since $2c^2=6d^2-2q\equiv -s^2-3r^2\equiv -s^2$ (mod 3) then we
have $2c^2\equiv 2s^2$ (mod 3) which together with
Lemma~\ref{primitivesol6kplusminus1} gives the desired conclusion.
\end{proof}

For $a=b=c=1$, we have 8 elements in $\mathrm{T}_1^0$ as shown in
Figure~3, one for each triangular face of a cuboctahedron
(Archimedean solid $A_1$ or also known as a truncated cube):

 The number of tetrahedra in $\mathrm{T}_m^0$ in general depends on $m$. For instance, if $m=3$
then we calculated that $|\mathrm{T}_3^0|=40$. Let us denote by
$TO(m)$ the cardinality of $\mathrm{T}_m^0$.

In general, the faces of an element in $\mathrm{T}_m^0$ must have
normal vectors that satisfy (\ref{planeeq}).  But these equations
can be very different and one has to go search far enough to find
such an example. For instance, the tetrahedron $OABC$ where $A=(376,
-841, 2265)$, $B=(-1005, -2116, 701)$, $C=(1411, -1965, 356)$ has
the four faces with normal vectors

\[\begin{array}{c}
 (-187, 113, 73),\ \ satisfying\ \ \ \    187^2+113^2+73^2=3(133^2),  \\
(-343, -253, -37), satisfying\ \ 343^2+253^2+37^2=3(247)^2,\\
(19, 41, 151), satisfying\ \  19^2+41^2+151^2=3(91)^2\ \ and \ \ \\
(391, -2461, 1661), satisfying\ \  \ 391^2+2461^2+1661^2=3(1729)^2.
\end{array}
\]

\begin{theorem}\label{main} Every tetrahedra in
$\mathrm{T}_{\ell}^0$ can be obtained by taking as one of its faces
an equilateral triangle having the origin as a vertex and the other
two vertices given by (\ref{paramone}) and (\ref{paramtwo}) with
$a$, $b$, $c$ and $d$ odd integers satisfying (\ref{planeeq}) with
$d$ a divisor of $\ell$, and then completing it with the fourth
vertex as in (\ref{fourthpoint}) for some $(m,n)\in
\Omega({\ell}/d)$.

Conversely, if we let $a$, $b$, $c$ and $d$ be a primitive solution
of (\ref{planeeq}), let $k\in \Bbb N$ and $(m,n)\in \Omega(k)$, then
the coordinates of the point $R$ in (\ref{fourthpoint}) are

(a) all integers, if $k\equiv 0$ (mod 3) regardless of the choice of
signs or

(b) integers, precisely for only one choice of the signs if $k\not
\equiv 0$ (mod 3).
\end{theorem}

\noindent \begin{proof} The first part of the theorem follows from
the arguments at the beginning of this section. For the second part,
the case in which $k$ is a multiple of $3$ follows from
Lemma~\ref{mandnequiv}.

Let us look into a stituation in which $k\not \equiv 0$ (mod 3).
First we are going to analyze the third coordinate of $R$. Since
(\ref{paramtwo}) gives $m_z+m_w=\frac{r-s}{2}+r=\frac{3r-s}{2}$ and
$n_z+n_w=\frac{3r+s}{2}$ we see that

\begin{equation}\label{eqonz}
2[(m_z+m_w)m-(n_z+n_w)n\pm 2ck]\equiv -s(m+n)\pm ck\ (mod \ 3),
\end{equation}

\noindent which by Lemma~\ref{mandnequiv} and
Lemma~\ref{primitivesol6kplusminus1} give that

$$-s(m+n)\pm ck\equiv
\pm k(s\pm c)\equiv 0\ \ (mod \ 3)$$

\noindent only for one choice of signs. So, the last coordinate of
$R$ satisfies the statement of our theorem in part (b). If the
coordinates of $R$ are $(x_{R},y_{R},z_{R})$ we must have
$x_{R}^2+y_{R}^2+z_{R}^2=\ell^2$. So, it is enough to analyze just
one other coordinate of $R$. From (\ref{paramtwo}) we obtain that
$2q(m_x+m_u)=-db(3r+s)-ac(r-s)-2(rac+dbs)\equiv acs$ (mod 3), and
$2q(n_x+n_u)=-2(rac+dbs)-db(s-3r)-ac(r+s)\equiv -acs$ (mod 3).
Therefore,

$$2q[(m_x+m_u)m-(n_x+n_u)n\pm 2ak]\equiv acs(m+n)\pm qak\ (mod \ 3).$$

But $q\equiv -1$ (mod 3) by Lemma~\ref{primitivesol6kplusminus1},
and so
$$2q[(m_x+m_u)m-(n_x+n_u)n\pm 2ak]\equiv a(cs(m+n)\mp k)\ (mod \ 3).$$

\noindent Because $c^2\equiv 1$ (mod 3), we can multiply the above
relation by $c$ to get

$$2qc[(m_x+m_u)m-(n_x+n_u)n\pm 2ak]\equiv -a[-s(m+n)\pm ck]\ (mod \ 3),$$

\noindent which shows, based on (\ref{eqonz}), that
$(m_x+m_u)m-(n_x+n_u)n\pm 2ak$ is divisible by $3$ if and only if
$z_R$ is an integer. Hence we have proved the last part of
Theorem~\ref{main}. \end{proof}

The next figure shows how our tetrahedra look if we start from
equilateral triangles as bases contained in a plane and this is as
expected according to Theorem~\ref{main}.

\begin{center}
\epsfig{file=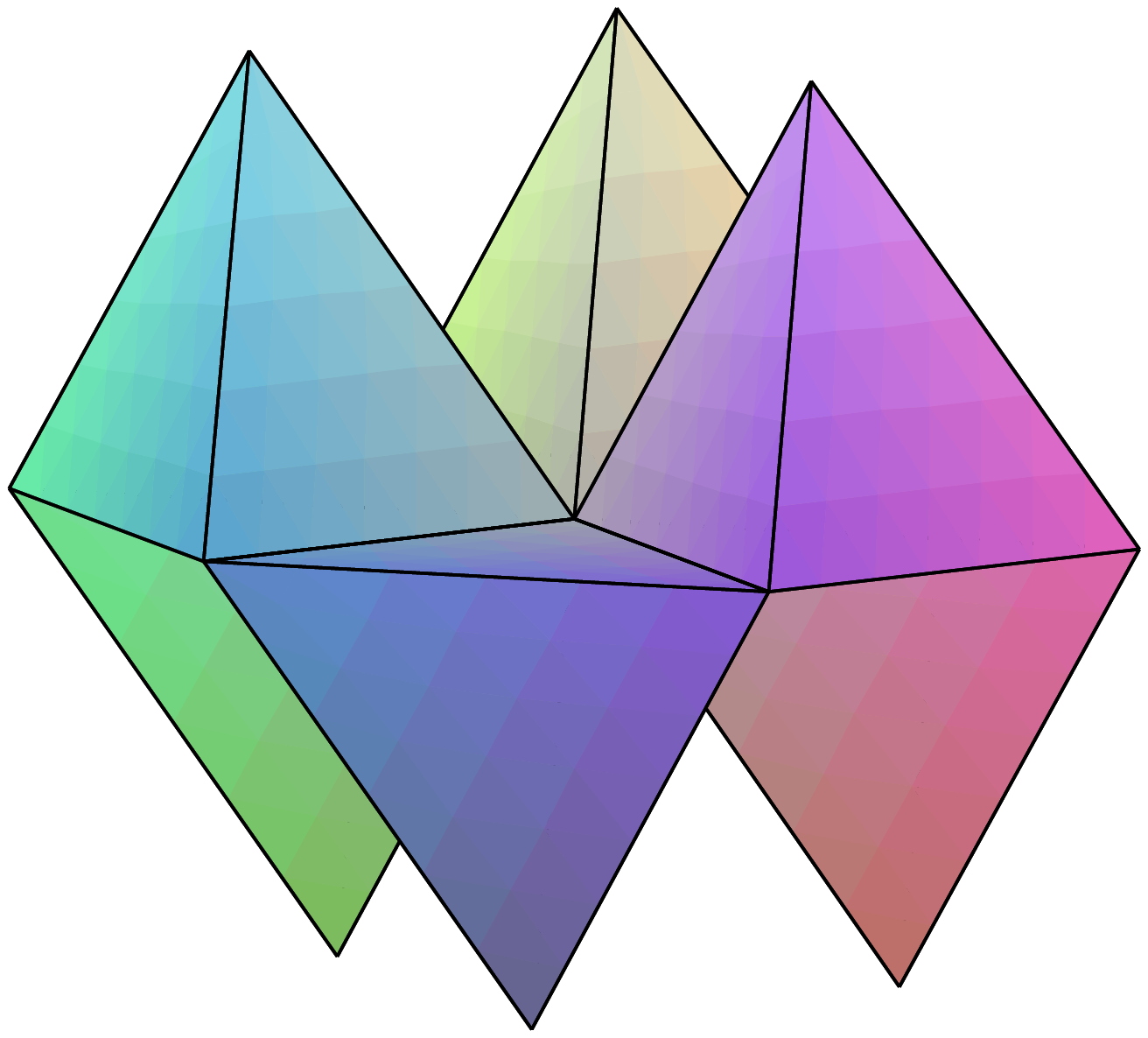,height=2.5in,width=3in}
\vspace{.1in}\\
{\it Figure 4: Regular tetrahedra in the case $k\not \equiv 0$ (mod
3)} \vspace{.1in}
\end{center}

{\bf Remark:} The alternating behavior that can be observed in
Figure~4 follows from Theorem~\ref{main} part (b). Indeed, we can
start with an equilateral triangle contained in a plane
$\mathfrak{P}$ which is one of the tiles and the generator of a
regular triangular tessellation.  By Theorem~\ref{main} part (b)
each such tile from this tessellation is the face of a regular
tetrahedron in $\mathrm T$ that is located in one and only one of
the sides of $\mathfrak{P}$. We denote the two half spaces by
$\mathfrak{P_{+}}$ and $\mathfrak{P_{-}}$. Let us take one such
tetrahedron, say in $\mathfrak{P_{+}}$, and translate this
tetrahedron such that the fourth vertex,  $R$, becomes the origin.
The other vertices give rise to three vertices of integer
coordinates in $\mathfrak{P_{-}}$ that are each the fourth vertex
$R'$ for three other tetrahedra living in $\mathfrak{P_{-}}$ as in
Figure~4. This type of translation can be repeated with one of the
tetrahedra in $\mathfrak{P_{-}}$ and one obtains the alternating
tetrahedra as Figure~4.

\begin{proposition}\label{normal}
For a regular tetrahedra in $\mathrm T_m$ each face lies in a plane
with a normal vector $n_i=(a_i,b_i,c_i)\in \mathbb {Z}^3$,
$i=1,2,3,4$, which satisfy
\begin{equation}\label{orth}
\begin{array}{c}
a_i^2+b_i^2+c_i^2=3d_i^2, \ \ d_i \ an\ odd\ integer\  dividing\ m, 1\le i\le 4\\ \\
a_ia_j+b_ib_j+c_ic_j+d_id_j=0, \ 1\le i<j\le 4
\end{array}
\end{equation}
\end{proposition}

\begin{proof} We are going to refer to the Figure~2.
Let $E$ be the center of the face $OPQ$, $F$ be the center of the
face $RPQ$, $G$ the midpoint of the side $\overline{PQ}$, and $H$
the center of the tetrahedron $ORPQ$ ($O$ being the origin). Since
$O$, $R$, $E$, $F$, $G$, and $H$ are all coplanar, and $F$ and $E$
are the intersections of the medians on each corresponding face we
have $\frac{RF}{RG}=\frac{2}{3}$ and $\frac{EG}{EO}=\frac{1}{2}$. By
Menelaus's theorem $\frac{HO}{HF}\cdot \frac{RF}{RG} \cdot
\frac{EG}{EO} =1$. This gives $\frac{HO}{HF}=3$ and from here we
have $HO=HR=\frac{3}{4}\sqrt{\frac{2}{3}}\ell=\sqrt{\frac{3}{8}}
\ell$. Using the law of cosines in the triangle ORH gives:
$OR^2=2HR^2-2HR^2\cos \widehat {OHR}$ which gives
\begin{equation}\label{angle}
\cos \widehat {OHR}= -\frac{1}{3}.
\end{equation}

So, if the normal exterior unit vectors of the faces RPQ and OPQ
are, say $\frac{(a_1,b_1,c_1)}{\sqrt{3}d_1}$ and
$\frac{(a_2,b_2,c_2)}{\sqrt{3}d_2}$, using (\ref{angle}) gives
$a_1a_2+b_1b_2+c_1c_2+d_1d_2=0$. The rest of the identities follow
similarly or from what we have mentioned so far.\end{proof}

{\bf Remark:} If we consider the matrix

\begin{equation}\label{matriceaortogonala}
MT:=\frac{1}{2}\begin{bmatrix}
                   \frac{a_1}{d_1} & \frac{b_1}{d_1} & \frac{c_1}{d_1} & 1 \\
                    \frac{a_2}{d_2} & \frac{b_2}{d_2} & \frac{c_2}{d_2} & 1 \\
                    \frac{a_3}{d_3} & \frac{b_3}{d_3} & \frac{c_3}{d_3} & 1 \\
                    \frac{a_4}{d_4} & \frac{b_4}{d_4} & \frac{c_4}{d_4} & 1 \\
                 \end{bmatrix}
\end{equation}

\noindent the restrictions in (\ref{orth}) can be reformulated in
terms of $MT$ by requiring it to be orthogonal, i.e.
$MT(MT)^t=(MT)^tMT=I$. It is known that the set of orthogonal
matrices form a group. We also want to point out that in
(\ref{orth}) only two of the exterior normal vectors $n_i$ are
essential. The other two can be computed from the given ones.

An interesting question here is whether or not every two vectors
$(a,b,c)$ and $(a',b',c')$ which satisfy the conditions

\begin{equation}\label{twonormals}
a^2+b^2+c^2=3d^2,\ \ \ a'^2+b'^2+c'^2=3d'^2,\ \ \ aa'+bb'+cc'+dd'=0
\end{equation}

\noindent for some $d,d'\in \mathbb Z$, corresponds to a tetrahedron
in $\mathrm T$ containing two faces normal to the given vectors. It
turns out that the answer is a positive one and the proof of it
follows from Theorem~4 in \cite{eji1}.

Another corollary of Theorem~\ref{main} about solutions of a
particular case of the Diophantine system (\ref{twonormals}) is
given next.

\begin{corollary}\label{corolary}
For every odd integer $d>1$, the Diophantine system

\begin{equation}\label{system}
a^2+b^2+c^2=3d^2,\ a'^2+b'^2+c'^2=3d^2,\ aa'+bb'+cc'=-d^2
\end{equation}
always has  nontrivial solutions (one trivial solution is $a=b=c=d$
and $a'=b'=-d$, $c=d$.)
\end{corollary}

\begin{proof} We have shown in  \cite{eji1} that the Diophantine equation
$a^2+b^2+c^2=3d^2$  always has nontrivial solutions. We then take
$m=n=1$ in Theorem~\ref{generalpar} to obtain an equilateral
triangle contained in the plane of normal vector $(a,b,c)$. By
Theorem~\ref{main}, case $k=1$, we can complete this equilateral
triangle to a regular tetrahedron in $\mathrm T$. Taking the normals
to the new faces will give four normals  and say
$(\hat{a},\hat{b},\hat{c})$ is one of them. Let us assume that
$gcd(\hat{a},\hat{b},\hat{c})=1$. We know that there is a $\hat{d}$
such that $ \hat{a}^2+\hat{b}^2+\hat{c}^2=3\hat{d}^2$. The side
length of this tetrahedron is, by our choice of $m$ and $n$, equal
to $d\sqrt{2}$. Therefore, from Theorem~\ref{generalpar} applied to
the plane normal to $(\hat{a},\hat{b},\hat{c})$, we get that
$d\sqrt{2}=\hat{d}\sqrt{2(u^2-uv+v^2)}$ for some $u,v\in \mathbb Z$,
which implies that $\hat{d}$ divides $d$. Hence we can adjust
$(\hat{a},\hat{b},\hat{c})$, if necessary, by a multiplicative
factor in order to satisfy (\ref{system}).
\end{proof}

\bigskip
\hrule
\bigskip

\noindent 2000 {\it Mathematics Subject Classification}: Primary
11A67;
Secondary 11D09, 11D04, 11R99, 11B99, 51N20.\\

\noindent {\it Keywords}: Diophantine equations, regular tetrahedra,
equilateral triangles, integers, parametrization, characterization.

\bigskip
\hrule
\bigskip

\noindent (Concerned with sequence \seqnum{A102698}.)

\bigskip
\hrule
\bigskip





\begin{thebibliography}{99}
\bibitem{g} R.~Guy, {\it Unsolved Problems in Number Theory}, Springer-Verlag,
2004.
\bibitem{gr} E.~Grosswald, {\it Representations of Integers as Sums of Squares}, Springer Verlag, 1985.
\bibitem{rceji} R.~Chandler and E.~J.~Ionascu, {\it A characterization of all equilateral triangles in $\Bbb Z^3$},
 arXiv:0710.0708v1, submitted December 2007.

\bibitem{eji1} E.~J.~Ionascu, {\it A parametrization of equilateral triangles having integer
coordinates}, Journal of Integer Sequences, {\bf 10}, (2007),
Article 07.6.7
\bibitem{eji2} E.~J.~Ionascu, {\it Counting all equilateral triangles in $\{0,1,...,n\}^3$},
to appear in Acta Mathematica Universitatis Comenianae



\bibitem{n} U.~P.~ Nair, {\it Elementary results on the binary
quadratic form $a^2+ab+b^2$}, arXiv:0408107v1
\bibitem{r} K.~Rosen, {\it Elementary Number Theory}, Fifth Edition, Addison Wesley, 2004
\bibitem{mw} C. J. Moreno and S.S. Wagstaff, JR. {\it Sums of
squares}, Chapman $\&$ Hall/CRC, 2006.


\bibitem{OL} Neil J.~A. Sloane, \newblock {\em The On-Line Encyclopedia of
Integer Sequences}, \newblock 2005, \newblock published
electronically at
  http://www.research.att.com/$\sim$njas/sequences/.
\end{thebibliography}
\end{document}